\documentclass[12pt]{amsart}%\documentclass[12pt,a4paper]{amsart}
\usepackage{amsfonts}

\usepackage[letterpaper]{geometry}

\usepackage{amssymb}
\def\R{\mathbb{R}}

\def\N{\mathbb{N}}

\newtheorem{thm}{\bf Theorem}[section]
\newtheorem{lemma}[thm]{\bf Lemma}
\newtheorem{prop}[thm]{\bf Proposition}

\theoremstyle{definition}

\begin{document}

\title{ON arithmetic sums of Cantor-type sequences of integers}

 \author[Norbert Hegyv\'ari]{Norbert Hegyv\'ari}
 \address{Norbert Hegyv\'{a}ri, ELTE TTK,
E\"otv\"os University, Institute of Mathematics, H-1117
P\'{a}zm\'{a}ny st. 1/c, Budapest, Hungary and associated member of Alfr\'ed R\'enyi Institute of Mathematics, Hungarian Academy of Science, H-1364 Budapest, P.O.Box 127.}
 \email{hegyvari@renyi.hu}

\begin{abstract}
We are looking for integer sets that resemble classical Cantor set and investigate the structure of their sum sets. Especially we investigate $FS(B)$ the subset sum of sequence type $B=\{\lfloor p^n\alpha\rfloor\}^\infty_{n=0}$. When $p=2$, then we prove $FS(B)+FS(B)=\N$ by analogy with the Cantor set, and some structure theorem for $p>2$.

MSC 2020 Primary 11B13,  Secondary 11B75, 11B25

\end{abstract}

 \maketitle

\section{Introduction}
The standard and generalized Cantor sets have a wide literature. Generally a set $C\subseteq \R$ is said to be Cantor set, if $C=J\setminus \cup_{i\geq 1} N_i$ where J is a bounded closed interval and $\{N_i\}$ is a countable -- finite or infinite --
system of disjoint open intervals contained in $J$. The standard (and well-known) Cantor set has many different definition. We now recall one of its properties that best fits our discussion. The Cantor set $C$ is a nested intersection of closed sets  in the form $C=\cap_{i\geq 1} C_i$, $C_1=[0,1]$, $C_2=[0,1/3]\cup[2/3,1]$ e.t.c. and $C_i\supset C_{i+1}$, $i\in \N$. Considering $C_n$ one can list in increasing order of the left open intervals $N_1$ (with length $1/3^{n-1}$), $N_2$ with length $1/3^{n-2}$, $N_1$ with length $1/3^{n-1}$ and so on. Listing the indeces of the open intervals gives 1,2,1,3,1,2,1,4,1,2,1,3,1,2,1,5,\dots which sequence is OEIS A001511 (Interestingly, the Cantor set is not included on this page). 

Another property of the Cantor set is the following: write $C=[0,1]\setminus \cup_{i\geq 1} N_i=\cap_{i\geq 1} C_i$. Let $N_i=(\beta_i, \gamma_i)$ be any gap in $C_n$, and assume that $[\alpha_i, \beta_i]$ is largest interval left adjacent to $N_i$ that
contains no gap of $C$ of length $\geq |N_i|$. Then $(C\cap [\alpha_i, \beta_i])+t_i\subseteq C$, where $t_i=\gamma_i-\alpha_i$ and $N_i+t_i\subseteq N_j=(\beta_j, \gamma_j)$, for some $j\in \N$. Call this property piecewise shift invariant. The finite prototype of this sequence is the Salem-Spencer set (see [SS]).

Let $A\subseteq \N$ and assume that $0,1\in A$. We say that $A$ is a Cantor-type sequence of integers, if $A=\N\setminus \cup_{i\geq 1} (x_i,y_i)$, where $(x_i,y_i)$, ($i\in \N)$ is nonempty interval (call it gap), and $A$ is piecewise shift invariant.

Let $P_1=\{0,1,k:k>2\}$,$P_2=\{0,1,2,3,k:k>4\}$, $P_3=\{0,1,\dots, 7,k; \ k>8\}$ and $P_4=\{0,1,\dots, r,k:r\geq 10; \ k>r+1\}$. 
The following inverse statement shows a relationship between Cantor sets and subset sums (see similar inverse results in [H1] and [N1]):

\begin{prop}{\it 
Let $A$ be a Cantor-type sequence if integers with some prefix $P_i$, $i=1,2,3$ or $4$. There exists a $B=\{b_1<b_2<\dots \}\subseteq \N$ such that $A=FS(B)$, where $FS(B):=\{\sum_{i=1}^\infty\varepsilon_ib_i: \ b_i\in B, \ \varepsilon_i \in \{0,1\}, \ \sum_{i=1}^\infty\varepsilon_i<\infty\}$ is a finite subset sum of $B$. 
Conversely if $B=\{1=b_1<b_2<\dots \}\subseteq \N$ is any sequence with $b_{i+1}> b_1+b_2+\dots b_i$ then $FS(B)$ is a  Cantor-type sequence of integers.}
\end{prop}
We postpone the proof to the end of the paper.
\section{Some arithmetic in Cantor sets}
A very suitable example of a  Cantor-type sequence of integers would be $C_{p,\alpha}:=FS(B)$, with $B=\{\lfloor p^n\alpha\rfloor\}^\infty_{n=0}$, where $p>0$ integer and $1<\alpha <2$ real number. This type of sequence has a big literature in arithmetic view when $p=2$ (see e.g. [EG], [H1], [L] [XFM]).

A very interesting result of the arithmetic of the Cantor set is that $C+C=[0,2]$ and as a corollary $C-C=[-1,1]$. A further nice result is that for $b\in [1/3,3]$ $C+bC=[0,1+b]$ see details in [ART].

For $p>4$, $C_{p,\alpha}+tC_{p,\alpha}=\{x_1+(p-1)x_2; x_1,x_2\in C_{p,\alpha}\}; \ t\in \N$ has zero asymptotic density, so there is no an obvious structure of it. Nevertheless we show that for $t=p-1$ and most of $\alpha$ it contains arbitrary long arithmetic progression. In the following theorem we state it in a quantitative form.

Let $w(s,N)$ be the (inverse) van der Waerden function, i.e. for every $s-$coloring the integers up to $N$ there is a monochromatic arithmetic progression with length $w(s,N)$. 
%Write $\alpha$ in base $p$, i.e. write $\alpha=\eta_0+\sum^\infty_{i=1}\eta_ip^{-i}$ where $\eta_i\in \{0,1,\dots p-1\}$ $\eta_0=1$. Let $U_\beta=\{\alpha: |\{1\leq i\leq n: \eta_i\neq 0\}|> \beta n \}$.
\begin{thm}
Let $p>1$ be and integer. For almost all $\alpha \in (1,2)$, $C_{p,\alpha}+(p-1)C_{p,\alpha}$ contains arithmetic progression with length $w(p-1, \lfloor c_N\frac{\log N}{\log p}\rfloor )$ up to $N$, where $c_N\to \frac{p-1}{p}$ as $N\to \infty$.
\end{thm} 
\begin{proof}[Proof of Theorem 2.1]
First note that for every sequence $X=\{x_n\}^\infty_{n=0}$ holds the following: if $v\in FS(X)$ then there is an $w\in FS(X)$ such that $s_n-v=w$, where $s_n=\sum_{i=0}^nx_i$. So let $FS(X)=C_p$ $N$ be an integers, big enough, and define $n$ by $s_n\in C_p$, $s_n\leq N<s_{n+1}$. Then $n\geq \frac{\log N}{\log p}$.

For ever $1\leq k\leq n$ define $\Delta_k:=x_k-(p-1)s_{k-1}$. Write $\alpha$ in base $p$; $\alpha=\sum^\infty_{i=0}\eta_ip^{-i}$, $\eta_i\in \{0,1,\dots ,p-1\}$ where there are infinitely many non-zero digits (for the sake of clarity). Recall that $\eta_0=1$.  
\begin{lemma}
For every $n\geq 0$, $\Delta_n=\sum^n_{i=0}\eta_i$
\end{lemma}
{\it Proof.} Since $x_n= \lfloor p^n\alpha\rfloor$, the recursion $x_{n+1}=px_n+ \eta_{n+1}$ follows. 

For $n=0$, $\Delta_0=x_0=\eta_0$. When $n=1$, $\Delta_1=x_1-(p-1)x_0=px_0+\eta_1-(p-1)x_0=x_0+\eta_1=\eta_0+\eta_1$. So inductively if $\Delta_n=\sum^n_{i=0}\eta_i$ then  
$$
\Delta_{n+1}=x_{n+1}-(p-1)s_n=px_n+\eta_{n+1}-(p-1)x_n-(p-1)s_{n-1}=$$
$$
=\Delta_n+\eta_{n+1}=\sum^{n+1}_{i=0}\eta_i. \quad \Box
$$
As we noted, for every $x_k\in C_p$, $k\leq n$ there exists an $x_m\in C_p$ such that $s_n-x_k=x_m$, which follows that 
$
s_n-(x_k-(p-1)s_{k-1})=x_m+(p-1)s_{k-1}=s_n-1-\sum^{k-1}_{i=0}\eta_i\in C_p+(p-1)C_p.
$
Write $y_k=s_n-1-\sum^{k-1}_{i=0}\eta_i$, $1\leq k\leq n$. Let $y_{k_1}>y_{k_2}>\dots >y_{k_m}$ be the longest  subsequence of 
$\{y_k\}$ for which $\eta_{k_i-1}\neq 0$. 
$1\leq y_{k_i}-y_{k_i+1}=\eta_{k_i}\leq p-1$. Since for almost all $\alpha$ the occurrence of each digits are the same (see e.g. in [K]), we have that $m=(1+o_N(1))\frac{(p-1)}{p}\frac{\log N}{\log p}$.

It is known that a sequence with bounded gaps contains arbitrary long arithmetic progression. We now use a finite version of this, giving, for seek completeness, a proof.  (The author did not find an explicit version of this statement in the literature).
\begin{lemma}
Let $Z=\{z_1<z_2<\dots<z_m\}$ be a sequence of integers and assume that $z_{i+1}-z_i\leq K$ for some $K>0$. Then $Z$ contains an arithmetic progression with length $w(K,\lfloor m/K\rfloor)$.
\end{lemma}
{\it Proof:} Since $Z$ is shift invariant we can assume that $z_1=0$. Consider the arithmetic progression $K,2K,\dots \lfloor m/K\rfloor K$. Let $\chi: Z\mapsto \{1,2,\dots K\}$ be a coloring of $Z$ as follows: let $z_s$ be the first element in the interval $[iK,(i+1)K)$. By the Dirichlet principle $Z\cap [iK,(i+1)K)\neq \emptyset$. Then let $\chi(i)=z_s-Ki$. It is a $K-$coloring of the arithmetic progression hence by the van der Waerden theorem it contains a monochromatic  -- say color $0\leq r<K-1$ -- arithmetic progression $P$. Then $P+r$ also an arithmetic progression that coincident with the elements of $Z$. $\Box$

Use the previous lemma for $\{y_{k_1}>y_{k_2}>\dots >y_{k_m}\}$ and $m$, completing the proof of the theorem.
\end{proof}

\medskip

When $p=2$, the density of $C_2$ is $1/\alpha$ so some structure theorem is available (see [B]). But the following theorem shows that the sumset of $C_2=FS(\{\lfloor 2^n\alpha\rfloor\}^\infty_{n=0})$ is an analogue of the classical Cantor set of $\R$, where $C+C=[0,2]$.
\begin{thm}
Let $1<\alpha<2$ and $C_2=FS(\{\lfloor 2^n\alpha\rfloor\}^\infty_{n=0})$. We have $C_2+C_2=\N$.
\end{thm}
\begin{proof}[Proof of Theorem 2.4]
We follow the idea of the proof of the previous theorem. Write $\alpha$ in base $2$; $\alpha=\sum^\infty_{i=0}\varepsilon_i2^{-i}$, $\varepsilon_i\in \{0,1\}$ where we assume again that there are infinitely many non-zero digits. 
Use the previous notation for $\Delta_k$ in case of $p=2$, i.e. $\Delta_k:=\lfloor 2^k\alpha\rfloor-s_{k-1}$. By Lemma 2.2 we have that for every $n$, $\Delta_n=\sum^n_{i=0}\varepsilon_i$.

We prove that for every $n=0,1,2,\dots$, $[0,s_n]\subseteq C_2+C_2$. 

Recall that $0\in C_2$ and write $a_n=\lfloor 2^n\alpha\rfloor$ and $C_2=\{x_t\}$. When $n=1$, $s_1=a_1=1$, so $[0,1]\subseteq C_2\subseteq C_2+C_2$. Now assume that $[0,s_{n-1}]\subseteq C_2+C_2$. Our task is to file the gap $[s_{n-1}+1,a_n-1]$ then $[0,s_{n-1}]+a_n=[0,s_n]\subseteq C_2+C_2$ which proves the theorem.

Let $x\in [s_{n-1}+1,a_n-1]$ and let $r=x-s_{n-1}$. We use that the digits are $0$ or $1$, hence $\{\Delta_k\}_{k=0}^m=[0,\sum^m_{i=0}\varepsilon_i]$ and $a_n-s_{n-1}=\sum^n_{i=0}\varepsilon_i$ so there is an $m$, $0\leq m\leq n$ for which $r=\sum^m_{i=0}\varepsilon_i=a_m-s_{m-1}$. There is an $t$ such that $s_n-s_{m-1}=x_t\in C_2$, thus 
$
s_n+a_m-s_{m-1}=x_t+a_m=s_n+r=x\in C_2+C_2.
$
\end{proof}

\begin{proof}[Proof of Proposition 1]

First we prove there are sets $B_1,B_2,B_3$ and $B_4$ of integers which the prefix of their subset sums are $P_1,P_2,P_3$ and $P_4$ respectively. 

For $P_1=\{0,1,k:k>2\}$, let $B_1=\{0,1,k:k>2\}$ and so $FS(B_1)=\{0,1,k; k>2\}$. For $P_2=\{0,1,2,3,k:k>4\}$ let $B_2=\{0,1,2, k>4\}$, $FS(B_2)=\{0,1,2,3,k; k>2\}$. 
Let $B_3=\{0,1,2,4, k;k>7\}$ then  $FS(B_3)=P_3=\{0,1,\dots, 7,k; \ k>8\}$. Finally let $P_4=\{0,1,\dots, r,k:r\geq 10; \ k>r+1\}$. Let us define $n$ by ${n+1\choose 2}\leq r< {n+2\choose 2}$. Let $B_4=\{0,1,\dots,n-1,n+s, k\}; k>r+1$, where $s=r-{n+1\choose 2}$. $FS(\{0,1,\dots,n-1\})=\{0,1,2,\dots {n\choose 2}\}$. 

Furthermore $FS(B_4)=FS(\{0,1,\dots,n-1\})+\{0,n+s\}$.
For $FS(B_4)$ to contain the entire interval $\{0,1,2,\dots, r\}$, it is sufficient that
$$
{n+1\choose 2}-n\geq {n+2\choose 2}-{n+1\choose 2}
$$
which holds when $n\geq 4$.

To complete the proof use that $A$ is a Cantor-type sequence, other words piecewise shift invariant which follows from the property
$$
FS(b_1,b_2,\dots,b_n)=FS(b_1,b_2,\dots,b_{n-1})+\{0,b_n\}.
$$
The reverse statement is obvious.
\end{proof}

%{\bf Data availability}

%No data was used for the research described in the article.

\section*{Acknowledgment}

The research is supported by the National Research, Development and Innovation Office NKFIH Grant No K-129335.

\end{document}